\documentclass{article}
\title{A Lagrangian Klein bottle you can't squeeze}
\author{Jonny Evans}
\usepackage[utf8]{inputenc}
\usepackage[T1]{fontenc}
\usepackage{tipa}
\usepackage{fixltx2e}
\usepackage{graphicx}
\usepackage{longtable}
\usepackage{float}
\usepackage{wrapfig}
\usepackage{rotating}
\usepackage{amsmath}
\usepackage{textcomp}
\usepackage{marvosym}
\usepackage{wasysym}
\usepackage{amssymb}
\usepackage{hyperref}
\usepackage{amsthm,amsmath,amsfonts,amscd}
\usepackage{tikz}
\usetikzlibrary{decorations.markings}
\usetikzlibrary{decorations.pathreplacing}
\usepackage{parskip}

\newcommand{\QQ}{\mathbb{Q}}
\newcommand{\RR}{\mathbb{R}}
\newcommand{\ZZ}{\mathbb{Z}}

\newcommand{\rp}[1]{\mathbf{RP}^{#1}}

\newcommand{\nog}{\textrm{\normalfont \textipa{\ng}}}

\begingroup
\makeatletter
\@for\theoremstyle:=definition,remark,plain\do{%
\expandafter\g@addto@macro\csname th@\theoremstyle\endcsname{%
\addtolength\thm@preskip\parskip
}%
}
\endgroup
\usepackage{graphicx}

\newtheorem{Theorem}{Theorem}[section]
\newtheorem{Lemma}[Theorem]{Lemma}
\newtheorem{Corollary}[Theorem]{Corollary}

\newtheorem{Conjecture}[Theorem]{Conjecture}
\theoremstyle{remark}
\newtheorem{Remark}[Theorem]{Remark}
\theoremstyle{definition}
\newtheorem{Question}[Theorem]{Question}
\newtheorem{Definition}[Theorem]{Definition}

\begin{document}
\maketitle
\begin{abstract}
Suppose you have a nonorientable Lagrangian surface \(L\) in a symplectic 4-manifold. How far can you deform the symplectic form before the smooth isotopy class of \(L\) contains no Lagrangians? I solve this question for a particular Lagrangian Klein bottle. I also discuss some related conjectures.
\end{abstract}

\section{Introduction}

Here are two overlapping questions:

\begin{Question}[Minimal nonorientable genus]\label{qun:min_genus}
Given a symplectic 4-manifold \((X,\omega)\) and a
\(\ZZ/2\)-homology class \(\beta\in H_2(X;\ZZ/2)\), what is the
minimal nonorientable genus of a nonorientable Lagrangian surface
\(L\subset X\) with \([L]=\beta\)?

\end{Question}
\begin{Question}[Nonsqueezing]\label{qun:squeeze}
Given a symplectic 4-manifold \((X,\omega)\) and a nonorientable
Lagrangian surface \(L\subset X\), how far can you deform \(\omega\)
in cohomology before there is no Lagrangian smoothly isotopic to
\(L\)?

\end{Question}
If \(L\) is orientable then these questions are less interesting: the
genus is determined by \([L]^2= - \chi(L)\) and, in Question
\ref{qun:squeeze}, it is necessary to deform \(\omega\) subject to the
cohomological condition \(\int_L[\omega]=0\). By contrast, if \(L\) is
nonorientable, we have \(H^2(L;\RR)=0\), which means that it is
possible to deform \(\omega\), keeping \(L\) Lagrangian, in such a way
that \([\omega]\) ranges over an open set in \(H^2(X;\RR)\).

I will give some general discussion of these questions in turn, then
give a concrete example of a Lagrangian Klein bottle for which
Question \ref{qun:squeeze} can be answered completely (Theorem
\ref{thm:klein}).

One running theme throughout the discussion is the use of {\em
visible} and {\em tropical} Lagrangians in almost toric 4-manifolds:
these provide a rich source of Lagrangian submanifolds coming
respectively from straight lines and tropical curves in integral
affine surfaces. I have found them useful for thinking about some of
the phenomena under discussion, and for formulating
conjectures. Visible Lagrangians were introduced in Symington's work
\cite{Symington}; tropical Lagrangians were introduced independently
by Mikhalkin \cite{Mikhalkin} and Matessi \cite{Matessi}.

\subsection{Acknowledgements}

I would like to thank Ivan Smith, Emily Maw, Georgios Dimitroglou
Rizell, and Daniel Cavey for helpful conversations around this
topic. My research is supported by EPSRC Grant EP/P02095X/2.

\section{The minimal genus question}

\subsection{Review}

\begin{Definition}
Define the {\em nonorientable genus} of the nonorientable surface
\(\#_k\rp{2}\) to be \(k\). Proposition 1.1 of \cite{DaiHoLi} shows
that any \(\ZZ/2\)-homology class in a symplectic 4-manifold can be
represented by some embedded nonorientable Lagrangian, so Question
\ref{qun:min_genus} has a well-defined answer, which I will
denote\footnote{{\ng} is the International Phonetic Alphabet symbol
for the ``ng'' sound.} by \(\nog(X,\omega,\beta)\).

\end{Definition}
\begin{Remark}
Audin \cite{Audin} showed that \[P_2(\beta)=\chi(L)=2-k\mod 4,\]
where \(P_2\) denotes the Pontryagin square operation and \(\chi\)
is the Euler characteristic. If you find a Lagrangian with
nonorientable genus \(k\) then you can perform a Hamiltonian finger
move locally to introduce pairs of intersections with index
difference 1 and then perform Polterovich surgery \cite{Polterovich}
on these self-intersections to get an embedded Lagrangian with
nonorientable genus \(k+4\). This means that the set of genera which
can be realised is \(\{\nog(X,\omega,\beta), \nog(X,\omega,\beta)+4,
\ldots\}\).

\end{Remark}
\begin{Remark}
The quantity \(\nog(X,\omega,\beta)\) is known in a small range of
cases, the lower bound being the principal difficulty.

\begin{enumerate}
\item When \(X\) satisfies \([\omega]\cdot c_1(X)>0\), we know that
\(\nog(X,\omega,0)=6\). This follows from Givental's construction
\cite{Givental} of a Lagrangian \(\#_6\rp{2}\) in the 4-ball and
from the fact, proved by Shevchishin \cite{Shev} that \(X\)
contains no nullhomologous Lagrangian Klein bottles (see also the
beautiful papers by Nemirovski \cite{Nem,NemirovskiKB}).

\item Let \(X_{a,b,c}\) be the blow-up of the 4-ball in three subballs
so that the symplectic areas of the exceptional spheres
\(E_1,E_2,E_3\) are \(a,b,c\). Shevchishin and Smirnov
\cite{ShevSmi} show that \(E_1+E_2+E_3\) contains a Lagrangian
\(\rp{2}\) if and only if the following inequalities all hold
\[a<b+c,\quad b<c+a,\quad c<a+b.\] They call these the {\em
symplectic triangle inequalities}. This gives the lower bound
\(\nog(X_{a,b,c},\omega,E_1+E_2+E_3) \geq 5\) when \(a,b,c\)
violate the triangle inequalities.

\end{enumerate}
\end{Remark}
\begin{Remark}\label{rmk:shevsmir}
After the fact, we see that there is a {\em tropical} or {\em almost
toric} motivation for the Shevchishin-Smirnov triangle
inequalities. The almost toric base diagram in Figure
\ref{fig:shevsmi} depicts the blow-up \(X_{a,b,c}\); the affine
lengths \(a,b,c\) indicated correspond to the sizes of the
exceptional spheres \(E_1,E_2,E_3\). In red you can see a tropical
curve; using the ideas of Mikhalkin \cite{Mikhalkin} and Matessi
\cite{Matessi}, we can construct a Lagrangian submanifold \(L\)
living over a (small thickening of a) tropical curve. This tropical
Lagrangian is diffeomorphic to \(\rp{2}\) if and only if the
inequalities all hold: the preimage of the point marked with
cross-hairs is a circle in \(L\) whose neighbourhood is a M\"{o}bius
strip.

\begin{figure}[htb]
\begin{center}
\begin{tikzpicture}
\filldraw[fill=gray!25,draw=none] (0,4/3) -- (0,4) -- (4,0) -- (4/3,0) -- cycle;
\draw[thick] (0,4) -- (0,4/3) -- (4/3,0) -- (4,0);
\draw[dashed] (1,3) -- (1,2) node {\(\times\)};
\draw[dashed] (3,1) -- (2,1) node {\(\times\)};
\draw[red,thick] (2,1) -- (1,1) -- (1,2);
\draw[red,thick] (1,1) -- (2/3,2/3);
\node[red,thick] at (2/3,2/3) {\(\oplus\)};
\draw[dotted] (4/3,0) -- (0,0) -- (0,4/3);
\node at (2/3,0) [below] {\(c\)};
\node at (0,2/3) [left] {\(c\)};
\draw[dotted] (0,2) -- (1,2) node [midway,above] {\(a\)};
\draw[dotted] (2,0) -- (2,1) node [midway,right] {\(b\)};
\begin{scope}[shift={(5,0)}]
\filldraw[fill=gray!25,draw=none] (0,2/3) -- (0,4) -- (4,0) -- (2/3,0) -- cycle;
\draw[thick] (0,4) -- (0,2/3) -- (2/3,0) -- (4,0);
\draw[dashed] (2/3,10/3) -- (2/3,2) node {\(\times\)};
\draw[dashed] (7/3,5/3) -- (2,5/3) node {\(\times\)};
\draw[red,thick] (2,5/3) -- (2/3,5/3) -- (2/3,2);
\draw[red,thick] (0,3/3) -- (2/3,5/3);
\draw[dotted] (2/3,0) -- (0,0) -- (0,2/3);
\node at (1/3,0) [below] {\(c\)};
\node at (0,1/3) [left] {\(c\)};
\draw[dotted] (0,2) -- (2/3,2) node [midway,above] {\(a\)};
\draw[dotted] (2,0) -- (2,5/3) node [midway,right] {\(b\)};
\end{scope}

\end{tikzpicture}
\end{center}
\caption{Almost toric base diagrams for \(X_{a,b,c}\) with a tropical curve in red. Left: The symplectic triangle inequalities and the associated tropical Lagrangian is diffeomorphic to \(\rp{2}\) (with the core circle of a cross-cap living over the point marked by the cross-hair symbol). Right: The symplectic triangle inequalities are violated and the associated tropical Lagrangian is diffeomorphic to a disc.}
\label{fig:shevsmi}
\end{figure}
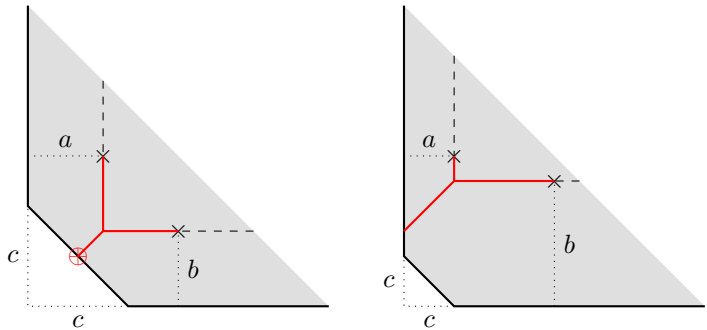

\end{Remark}
\subsection{\(S^2\times S^2\)}

Let \(X=S^2\times S^2\). Modulo an overall scale factor, any
symplectic form on \(X\) is diffeomorphic to one from the family
\(\lambda p_1^*\sigma+p_2^*\sigma\), where \(p_1,p_2\colon X\to S^2\)
are the two projections and \(\sigma\) is an area form on \(S^2\). We
know that \(\nog(X,\omega,0)=6\), which leaves two interesting
\(\ZZ/2\)-homology classes up to diffeomorphism: \(\beta=[\star\times
S^2]\) and the class \(\Delta\) of the diagonal. The Pontryagin
squares are \(P_2(\beta)=0\) and \(P_2(\Delta)=2\), so there is a
chance to represent \(\beta\) by Lagrangian Klein bottles.

\begin{figure}[htb]
\begin{center}
\begin{tikzpicture}
\fill[fill=gray!25,draw=black,thick] (-1.5,-1) -- (1.5,-1) -- (1.5,1) -- (-1.5,1) -- cycle;
\draw[thick,red] (-1.5,-0.75) -- (1.5,0.75);
\node at (0,0) [above] {\(\ell\)};
\node at (0,-1) [below] {\(\lambda\)};
\node[red,thick] at (-1.5,-0.75) {\(\otimes\)};
\node[red,thick] at (1.5,0.75) {\(\otimes\)};

\end{tikzpicture}
\end{center}
\caption{A visible Lagrangian Klein bottle in \((S^2\times S^2,\omega_\lambda)\) for \(\lambda<2\). The cores of two cross-caps are indicated with cross-hairs.}
\label{fig:visible_lag}
\end{figure}
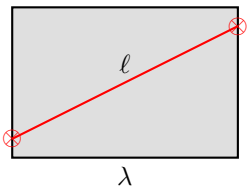

\begin{Lemma}\label{lma:klein}
If \(\lambda<2\) then \(\beta\) is represented by a Lagrangian Klein
bottle.
\end{Lemma}
\begin{proof}
The rectangle in Figure \ref{fig:visible_lag} is the toric moment
polygon for the standard Hamiltonian torus action on \(S^2\times
S^2\) with symplectic form \(\omega_\lambda\). There is a Lagrangian
Klein bottle living over the line \(\ell\) (slope \(1/2\)) in the
diagram. To see this, consider the two \(S^2\) factors sitting
inside \(\RR^3\) and let \((p_j,\theta_j)\) be cylindrical
coordinates on the \(j\)th factor (\(j=1,2\)). These are
action-angle coordinates, so \(\omega=\sum dp_j\wedge
d\theta_j\). The line \(\ell\) is given by \(2p_2=p_1\) and the
Lagrangian Klein bottle is cut out by this equation together with
\(\theta_2=-2\theta_1\). This is certainly Lagrangian for this
symplectic form. To see that \(L\) is a Klein bottle, notice that
the regular level sets of \(p_1\) restricted to \(L\) are circles
\(\theta_2=-2\theta_1\) in the \((\theta_1,\theta_2)\)-torus, which
collapse \(2\)-to-\(1\) onto the circles of maxima and minima at
\(p_1=\pm \lambda\) (as the torus collapses to the circle with
coordinate \(\theta_2\)). The projections of these circles are
denoted with cross-hairs in Figure \ref{fig:visible_lag}. \qedhere

\end{proof}
\begin{Remark}\label{rmk:lma:klein}
This \(L\) is a {\em visible Lagrangian} in the sense of Symington
\cite{Symington} as well as being a tropical Lagrangian in the sense
of Matessi \cite{Matessi} and Mikhalkin \cite{Mikhalkin}. This Klein
bottle is well-known: it appears in \cite{CastroUrbano} as a
Hamiltonian minimal Lagrangian, in \cite{Goodman} as a Hamiltonian
suspension, and in \cite{DaiHoLi} as a fibre connect-sum of
\(\rp{2}\)s. It has minimal Maslov number 1 and has a monotone
representative in its Lagrangian isotopy class if \(\lambda=1\).

\end{Remark}
If \(\lambda\geq 2\) then the line \(\ell\) does not fit into the
rectangle. The following conjecture seems natural; while I cannot
prove it, it inspired Theorem \ref{thm:klein} below.

\begin{Conjecture}\label{conj:kb}
There is no Lagrangian Klein bottle in the class \(\beta\) if
\(\lambda\geq 2\).

\end{Conjecture}
It is interesting to consider what happens for large \(\lambda\). We
have essentially no tools to prove lower bounds when the Lagrangians
are of high genus and may be Floer-theoretically obstructed. The most
pessimistic conjecture is that Lagrangians with high genus become
flexible enough that:

\begin{Conjecture}\label{conj:large_lambda}
\(\lim_{\lambda\to\infty}\nog(X,\omega_\lambda,\beta)<\infty\).

\end{Conjecture}
The following lemma gives an upper bound on
\(\nog(X,\omega_\lambda,\beta)\), but it goes to infinity with
\(\lambda\).

\begin{Lemma}\label{lma:trop}
We have \(\nog(X,\omega_\lambda,\beta)\leq 20\ell+2\) when
\(\lambda<10\ell+2\).
\end{Lemma}
\begin{proof}
If \(\lambda<10\ell+1\) then there is a tropical Lagrangian in the
class \(\beta\) with nonorientable genus \(20\ell+2\). We show the
tropical curve for \(\ell=2\) in Figure \ref{fig:trop_diag} below;
for general \(\ell\) we simply repeat the pattern between the
vertical blue bars as often as required to get from the left-hand
side to the right-hand side of the rectangle.

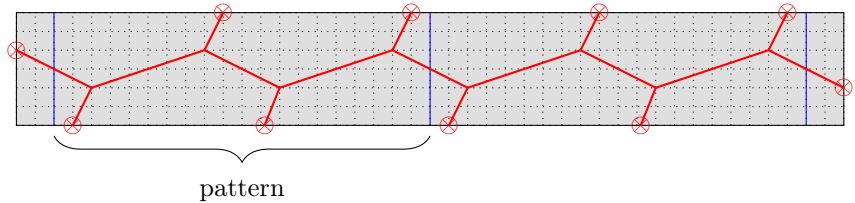
\begin{figure}[htb]
\begin{center}
\begin{tikzpicture}
\begin{scope}[scale=1.5]
\filldraw[fill=gray!25,draw=black] (0,0) -- (22/3,0) -- (22/3,1) -- (0,1) -- cycle;
\draw[step=1/6,dotted] (0,0) grid (22/3,1);
\draw[blue] (1/3,0) -- (1/3,1);
\draw[blue] (11/3,0) -- (11/3,1);
\draw[blue] (21/3,0) -- (21/3,1);
\begin{scope}[scale=1/3]
\node[red,thick] at (0,2) {\(\otimes\)};
\node[red,thick] at (1.5,0) {\(\otimes\)};
\node[red,thick] at (6.6,0) {\(\otimes\)};
\node[red,thick] at (5.5,3) {\(\otimes\)};
\node[red,thick] at (10.5,3) {\(\otimes\)};
\draw[thick,red] (0,2) -- (2,1) -- (5,2) -- (7,1) -- (10,2) -- (12,1);
\draw[thick,red] (2,1) -- (1.5,0);
\draw[thick,red] (7,1) -- (6.6,0);
\draw[thick,red] (5,2) -- (5.5,3);
\draw[thick,red] (10,2) -- (10.5,3);
\draw [decorate,decoration={brace,amplitude=10pt,mirror,raise=4pt},yshift=0pt] (1,0) -- (11,0) node [black,midway,below=0.6cm] {pattern};
\end{scope}
\begin{scope}[scale=1/3,shift={(10,0)}]
\node[red,thick] at (12,1) {\(\otimes\)};
\node[red,thick] at (1.5,0) {\(\otimes\)};
\node[red,thick] at (6.6,0) {\(\otimes\)};
\node[red,thick] at (5.5,3) {\(\otimes\)};
\node[red,thick] at (10.5,3) {\(\otimes\)};
\draw[thick,red] (2,1) -- (5,2) -- (7,1) -- (10,2) -- (12,1);
\draw[thick,red] (2,1) -- (1.5,0);
\draw[thick,red] (7,1) -- (6.6,0);
\draw[thick,red] (5,2) -- (5.5,3);
\draw[thick,red] (10,2) -- (10.5,3);
\end{scope}
\end{scope}

\end{tikzpicture}
\end{center}
\caption{A tropical curve giving a Lagrangian of genus \(20\ell+2\) in the case \(\ell=2\).}
\label{fig:trop_diag}
\end{figure}

The edges of this tropical curve are:
\begin{itemize}
\item internal edges parallel to either \((3,1)\) or \((2,-1)\),
\item external edges parallel to \((2,-1)\) or \((1,2)\).
\end{itemize}
The corresponding tropical Lagrangian intersects the horizontal
spheres with even multiplicity and the vertical spheres with odd
multiplicity, so it inhabits the class \(\beta\). The vertices of
the tropical curve are not smooth\footnote{At each vertex of a
tropical curve, the outgoing edges \(v_1,v_2,v_3\) must sum to zero;
if we write \(m\) for the determinant \(|v_1\wedge v_2|=|v_2\wedge
v_3|=|v_3\wedge v_1|\) then the self-intersection of this vertex is
defined to be \(\frac{m-1}{2}\). Smoothness means all vertices have
self-intersection zero.}: each has self-intersection equal to 2. By
{\cite[Theorem 3.2]{Mikhalkin}}, this tropical curve therefore
yields an immersed Lagrangian with \(8\ell\) double points and
\(2+4\ell\) cross-caps where it hits the boundary (marked with
cross-hairs in Figure \ref{fig:trop_diag}). When we perform
Polterovich surgery at the double points, we obtain a Lagrangian
which is topologically a surface of genus \(8\ell\) with \(4\ell+2\)
cross-caps. This has Euler characteristic
\(2-16\ell-4\ell-2=-20\ell\), so the nonorientable genus is
\(2+20\ell\). \qedhere

\end{proof}
\begin{Remark}
It seems harder to make the genus significantly smaller using
tropical Lagrangians, but there is no reason to believe that
tropical Lagrangians should give a sharp upper bound for \(\nog\).

\end{Remark}
\section{Nonsqueezing}

\subsection{Statement}

For each connected open interval \(I\subset\RR\) (length \(|I|\)), let
\(C_I\) denote the cylinder \(I\times(\RR/2\pi\ZZ)\) with coordinates
\((p,\theta)\), equipped with the symplectic form
\(\frac{1}{2\pi}dp\wedge d\theta\); this has total area \(|I|\). Let
\(S^2\) denote the 2-sphere equipped with its area form \(\sigma\)
satisfying \(\int_{S^2}\sigma=2\).

Let \(U_I=S^2\times C_I\). Note that \(U_I\) is obtained from
\((S^2\times S^2,\omega_{|I|})\) by excising the spheres
\(S^2\times\{n,s\}\), where \(n,s\) denote the poles of the second
factor. Arguing as in Lemma \ref{lma:klein}, we see that if \(|I|>1\),
the only nontrivial class \(\beta\in H_2(U_I;\ZZ/2)\) is represented
by a Lagrangian Klein bottle (see Figure \ref{fig:kb_reprise}).

\begin{figure}[htb]
\begin{center}
\begin{tikzpicture}
\filldraw[fill=gray!25,draw=none] (-1.5,-1) -- (1.5,-1) -- (1.5,1) -- (-1.5,1) -- cycle;
\draw[thick,black] (-1.5,-1) -- (-1.5,1);
\draw[thick,black] (1.5,-1) -- (1.5,1);
\draw[thick,red] (-1.5,-0.75) -- (1.5,0.75);
\node at (0,-1) [below] {\(2\)};
\node at (-1.5,0) [left] {\(|I|\)};
\node[red,thick] at (-1.5,-0.75) {\(\otimes\)};
\node[red,thick] at (1.5,0.75) {\(\otimes\)};

\end{tikzpicture}
\end{center}
\caption{The visible Lagrangian Klein bottle in \(U_I\) when \(|I|>1\).}
\label{fig:kb_reprise}
\end{figure}
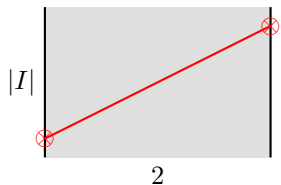

\begin{Theorem}\label{thm:klein}
Suppose that \(|I|\leq 1\). If \(\iota\colon K\to U_I\) is a
Lagrangian embedding of the Klein bottle in the class \(\beta\) then
\(\iota_*\colon \QQ=H_1(K;\QQ)\to H_1(U_I;\QQ)=\QQ\) is the zero
map.

\end{Theorem}
\begin{Remark}
The proof of Theorem \ref{thm:klein} will occupy the rest of the
paper. It uses SFT and neck-stretching.

\end{Remark}
\begin{Remark}
Note that if \(|I|>1\) then \(H_1(L;\QQ)\to H_1(U_{I};\QQ)\) is an
isomorphism for the Lagrangian Klein bottle \(L\) coming from Lemma
\ref{lma:klein}. To see this, take either one of the circles living
over the points marked with cross-hairs in Figure
\ref{fig:kb_reprise}; this is a generator for both \(H_1(L;\QQ)\)
and \(H_1(U_I;\QQ)\). We deduce:

\end{Remark}
\begin{Corollary}
The Lagrangian Klein bottle in \(U_{\left(0,1+\epsilon\right)}\)
from Lemma \ref{lma:klein} cannot be squeezed into \(U_{(0,1)}\).

\end{Corollary}
\begin{Remark}
To reduce Conjecture \ref{conj:kb} to this result, you would need to
produce a pair of symplectic spheres in the class \([S^2\times
\star]\) which ``link'' your Lagrangian Klein bottle in an
appropriate way. Since this class has non-minimal symplectic area,
it is difficult to control the SFT limit of such spheres.

\end{Remark}
We now proceed to the proof of Theorem \ref{thm:klein}.

\subsection{Mohnke's almost complex structure}

Pick a flat metric \(g\) on the Klein bottle. There is a contact form
(the canonical 1-form) on the unit cotangent bundle \(M\subset T^*K\)
whose closed Reeb orbits correspond to closed geodesics on \(K\). We
will not distinguish notationally between geodesics and the
corresponding Reeb orbits and we will write \(-\gamma\) for the
geodesic obtained by reversing \(\gamma\). There are two isolated
simple geodesics \(\gamma_0,\gamma_1\) which are the core circles for
two disjoint embedded M\"{o}bius strips in \(K\). Any isolated
geodesic is a multiple cover of one of these and all other geodesics
occur in one-parameter families. We call the isolated geodesics {\em
odd} and the other geodesics {\em even}.

\begin{Theorem}[Mohnke {\cite[Section 2.1]{Mohnke}}]\label{thm:mohnke}
There exists an almost complex structure \(J^-\) on the cotangent
bundle \(T^*K\) with the following properties:

\begin{enumerate}
\item \(J^-\) is cylindrical at infinity and suitable for
neck-stretching.

\item For any geodesic \(\gamma\) there is a finite-energy
\(J^-\)-holomorphic cylinder \(f_\gamma\) in \(T^*K\) asymptotic
to \(\gamma\) and \(-\gamma\).

\item {\cite[Lemma 7(2)]{Mohnke}} Any \(J^-\)-holomorphic cylinder in
\(T^*K\) which intersects the zero-section is one of these
\(f_\gamma\) for some closed geodesic \(\gamma\).

\end{enumerate}
\end{Theorem}
\begin{Remark}\label{rmk:intersection}
If we let \(W:=\overline{T^*K}\) denote the compactification of the
cotangent bundle obtained by gluing on its ideal contact boundary
\(M\) then there is a well-defined intersection pairing
\(H_2(W,M;\ZZ/2)\otimes H_2(W;\ZZ/2)\to\ZZ/2\). The cylinders
\(f_\gamma\) define elements of \(H_2(W,M;\ZZ/2)\) and we have
{\cite[Lemma 7(3)]{Mohnke}} \[f_\gamma\cdot K=\begin{cases} 1\mbox{
if }\gamma\mbox{ is odd}\\ 0\mbox{ if }\gamma\mbox{ is
even.}\end{cases}\]

\end{Remark}
\begin{Remark}[{\cite[Lemma 7(1)]{Mohnke}}]\label{rmk:noplanes}
Note that there are also no finite energy planes in \(T^*K\), nor in
the symplectisation \(\RR\times M\), for {\em any} cylindrical
almost complex structure adapted to our chosen contact form. This is
because there are no contractible Reeb orbits, and a finite energy
plane would provide a nullhomotopy of its asymptote.

\end{Remark}
\subsection{Neck-stretching}

Let \(I=(0,1)\) and \(\bar{I}=[0,1]\). Suppose there is a Lagrangian
Klein bottle \(K\subset U_{I}\) such that \(\QQ=H_1(K;\QQ)\to
H_1(U_{I};\QQ)=\QQ\) is nonzero (in particular, it is
injective). Think of \(K\) sitting inside \(U_{\bar{I}}\) and make
symplectic cuts to \(U_{\bar{I}}\) at \(p=0,1\) to obtain a Lagrangian
Klein bottle \(K\) living in the manifold \(X=S^2\times S^2\) equipped
with the product symplectic form giving the factors areas \(2\) and
\(1\) respectively. Crucially, the symplectic cut introduces
symplectic spheres \(S_0\) and \(S_1\) (at the \(p=0,1\) cuts
respectively) which are disjoint from \(K\).

Pick a sequence of almost complex structures \(J_t\), \(t\in\RR\), on
\(X\) with the following properties:
\begin{itemize}
\item on a Weinstein neighbourhood of \(K\), \(J_t\) coincides with
Mohnke's almost complex structure \(J^-\);
\item on a neck-stretching region \((a_t,b_t)\times M\) around \(K\),
\(J_t\) is a neck-stretching sequence;
\item the spheres \(S_0,S_1\) are \(J_t\)-holomorphic for all
\(t\in\RR\).
\end{itemize}
Pick a point \(k\) on \(K\) which does not lie on any of the cylinders
\(f_{\gamma}\) for an odd geodesic \(\gamma\). Let \(u_t\colon S^2\to
X\) be a \(J_t\)-holomorphic curve representing the class
\(\alpha=[\star\times S^2]\) and such that \(u_t(0)=k\); there is a
unique such \(u_t\) up to reparametrisation by a theorem of Gromov
{\cite[2.4.C]{Gromov}}, since \(\alpha\) is a minimal area sphere
class in \(X\).

By the SFT compactness theorem \cite{BEHWZ}, there is a sequence
\(t_i\) such that \(u_{t_i}\) converges (after reparametrisations) to
a holomorphic building with components in \(T^*K\) (the completion of
the Weinstein neighbourhood of \(K\)), components in \(\RR\times M\)
(the completion of the neck) and components in \(X\setminus K\) (the
completion of the complement of the Weinstein neighbourhood).

\subsection{SFT limit analysis}

The components \(v_1,\ldots,v_n\) of the SFT limit building living in
\(X\setminus K\) can be compactified, yielding topological surfaces in
\(X\) with boundary on \(K\); we will still denote these by
\(v_1,\ldots,v_n\). The sum of the \(\omega\)-areas of the \(v_i\)
(weighted by multiplicities if the SFT limit involves a branched
cover) equals the \(\omega\)-area of \(\alpha\), which is \(1\).

\begin{Lemma}\label{lma:atleasttwoplanes}
There must be at least two planar components amongst the \(v_i\),
possibly geometrically indistinct (i.e. having the same image).
\end{Lemma}
\begin{proof}
First note that the limit building intersects \(K\) because
\(u_t(0)=k\in K\) for all \(t\). It also necessarily has at least
one component in \(X\setminus K\) because \(T^*K\) is exact and so
contains no closed holomorphic curves. A genus zero holomorphic
building with at least two levels must have two planar components
(just for topological reasons) though these could be geometrically
indistinct. Any planar components live in \(X\setminus K\). \qedhere

\end{proof}
\begin{Lemma}\label{lma:exactlytwoplanes}
There are two components \(v_0,v_1\) of the limit building such that
\(v_i\cdot S_j=\delta_{ij}\). These components are planar and there
are no further components of the limit building in \(X\setminus L\).
\end{Lemma}
\begin{proof}
Since \(\alpha\) intersects \(S_0\) and \(S_1\) there must be
components of the limit building which intersect \(S_0\) and
\(S_1\). By positivity of intersections, either:

\begin{enumerate}
\item [(A)] there is one component \(v_1\) which hits both \(S_0\) and
\(S_1\) once transversely and all other components are disjoint
from \(S_0,S_1\).

\item [(B)] there are two components \(v_0,v_1\) such that \(v_0\)
intersects \(S_0\) once transversely and is disjoint from \(S_1\)
and vice versa for \(v_1\).

\end{enumerate}
Moreover, each of these components occurs with multiplicity one in
the SFT limit in order to get the correct intersection numbers
\(\alpha\cdot S_0,\alpha\cdot S_1\).

If \(v_2\) is a component which does not intersect \(S_0\) or
\(S_1\) then it defines a class in \(H_2(U_{I},K;\ZZ)\). By
assumption, the kernel of the map \(\ZZ\oplus\ZZ/2=H_1(K;\ZZ)\to
H_1(U_{I};\ZZ)=\ZZ\) is precisely the torsion part. Therefore the
long exact sequence \[\cdots\to H_2(U_{I};\ZZ)\to
H_2(U_{I},K;\ZZ)\to H_1(K;\ZZ)\to H_1(U_I;\ZZ)\to\cdots\] splits off
a sequence \[\cdots\to\ZZ\to H_2(U_{I},K;\ZZ)\to\ZZ/2\to 0.\] This
implies that the areas of classes in \(H_2(U_{I},K;\ZZ)\) are
half-integer multiples of the area of the generator \(\beta\in
H_2(U_{I};\ZZ)\), which is \(2\). Therefore \(v_2\) has integer
area.

The area of \(\alpha\) is \(1\), so the (weighted) sum of the
\(\omega\) areas of the \(v_i\) equals 1. Since \(v_1\) has positive
area, \(v_2\) must have positive area strictly less than \(1\), but this is
not possible if \(v_2\) has integer area. Therefore there cannot be
any component \(v_2\) disjoint from \(S_0\) and \(S_1\).

By Lemma \ref{lma:atleasttwoplanes}, there are at least two planar
components (or one planar component with multiplicity two) in the
limit building. This is not compatible with Case (A), so we must be
in Case (B) and \(v_0,v_1\) must additionally be planes. \qedhere

\end{proof}
\begin{Lemma}
\begin{enumerate}
\item All the remaining parts of the limit building are cylinders.

\item At least one of these cylinders lives in \(T^*K\) and has the form
\(f_{\gamma}\) for an odd geodesic \(\gamma\).

\item There are no other cylindrical components of the SFT limit
building in \(T^*K\).
\end{enumerate}
\end{Lemma}
\begin{proof}
\begin{enumerate}
\item If a component has three or more punctures then the limit building
must contain at least three planar components (counted with
multiplicity) but we have seen that all the planar components must
live in \(X\setminus K\) (Remark \ref{rmk:noplanes}) and that
there are precisely two such components (Lemma
\ref{lma:exactlytwoplanes}).

\item Since \(u_t(0)=k\) for all \(t\), the limit building contains a
component in \(T^*K\), which must be a cylinder of the form
\(f_\gamma\) by Theorem \ref{thm:mohnke}(3). At least one of these
cylindrical components must correspond to an odd geodesic because
\(\alpha\) has odd intersection with \(K\) in \(H_2(X;\ZZ/2)\) and
the intersection number picks up contributions from each component
of the building inside \(T^*K\), which are nontrivial if and only
if \(\gamma\) is odd (Remark \ref{rmk:intersection}).

\item If there are two or more cylindrical components in \(T^*K\) then
there must be a further cylindrical component in \(\RR\times M\)
which connects the asymptotes of two of these cylinders. Since
this cylinder has no positive asymptote, this cannot exist by the
maximum principle. \qedhere

\end{enumerate}
\end{proof}
\begin{proof}[Proof of Theorem \ref{thm:klein}]\label{prf:thm:klein}
We chose \(k\in K\) not to lie on any of the cylinders \(f_\gamma\)
for \(\gamma\) an odd geodesic, but we have showed that these are
the only cylinders which can arise as components of the SFT limit
building. Since the SFT limit building must pass through \(k\), we
get a contradiction. \qedhere

\end{proof}
\bibliographystyle{plain}
\bibliography{klein}
\end{document}